\documentclass[10pt]{paper}%
\usepackage[T1]{fontenc}
\usepackage{amsmath,amssymb}
\usepackage{fullpage}
\usepackage{amsthm}
\usepackage{amsmath}
\usepackage{graphicx}
\usepackage{multirow}
  \usepackage[all]{xy}
  \usepackage[all]{xy}
\usepackage{youngtab}
\usepackage{young}
\usepackage{slashbox}
\usepackage{amsfonts}
\usepackage{amssymb}%
\providecommand{\U}[1]{\protect\rule{.1in}{.1in}}
\providecommand{\U}[1]{\protect\rule{.1in}{.1in}}
\providecommand{\U}[1]{\protect\rule{.1in}{.1in}}
\providecommand{\U}[1]{\protect\rule{.1in}{.1in}}
\providecommand{\U}[1]{\protect\rule{.1in}{.1in}}
\newcommand{\ulambda}{{\boldsymbol{\lambda}}}
\newcommand{\umu}{{\boldsymbol{\mu}}}
\newcommand{\unu}{{\boldsymbol{\nu}}}
\newcommand{\uemptyset }{{\boldsymbol{\emptyset}}}

\newtheorem{Th}{Theorem}[section]

\newtheorem{Prop}[Th]{Proposition}

\theoremstyle{remark}
\newtheorem{Rem}[Th]{Remark}{\rmfamily}
\theoremstyle{definition}
{\rmfamily}
\newtheorem{Exa}[Th]{Example}{\rmfamily}

\newtheorem{abs}[Th]{\bfseries}

\begin{document}
\title{On the one dimensional representations  of  Ariki-Koike algebras at roots of unity}
\author{Nicolas Jacon}
\maketitle
\date{}
\begin{abstract}
We study the natural labeling of the one dimensional representations for Ariki-Koike algebras at roots of unity.
 For Hecke algebras of types $A$ and $B$, some of these representations can be identified with the socle of the Steinberg representation of a finite reductive group. We here give closed formulas for them. 
This uses, in particular,  several results concerning crystal isomorphisms and the Mullineux involution. 

\end{abstract}

\section{Introduction}
The Hecke algebras of complex reflection groups can be seen as natural deformations of complex reflection groups. In particular they both generalize  the Hecke algebras of type $A$ (deforming the symmetric groups) and the Hecke algebras of type $B$ (deforming the hyperoctohedral groups). The representation theory of these algebras has been well studied during the last  past decades. It appears to be quite  deep and related to the representation theory of various important algebraic objects (such as the quantum affine algebras, the quiver-Hecke algebras or the rational Cherednik algebras.) In type $A$ and $B$, a motivation for studying these representations comes from the representation theory of finite reductive groups where these Hecke algebras appear as endomorphism algebras of  permutation representations. 

Recently, a question on the combinatorial representation theory of these algebras has been  asked in a work of Meinolf Geck on the Steinberg representation of a finite reductive group $G$ of ``classical type'' over a field $L$ \cite{ST}. In the case where $[G:B]1_L=0$ (where $B$ is the Borel subalgebra of $G$),  this remarkable representation of $G$ is reducible in general but  one can show that its socle is simple. Using the so called ``Green correspondence'', Meinolf Geck has shown that one can identify this socle with the sign representation of the Hecke algebra (associated to the datum of $G$ and $B$) at a root of $1$. Now, a classification of the simple modules for these algebras is available. As the sign representation is one dimensional and thus irreducible,  the question of describing the above socle reduces to the problem of finding the precise labeling of the sign representation in terms of the known classification.

The aim of this work is to give, more generally, a complete description of all  the one dimensional representations in the wider context of  the Hecke algebra of the complex reflection group $G(l,1,n)$ where $l\in \mathbb{Z}_{>0}$ and $n\in \mathbb{Z}_{>0}$. This algebra (also known as Ariki-Koike algebra)   is defined as follows. Let $q$ be an indeterminate and $Q_1$, \ldots, $Q_l$ be an $l$-tuple of indeterminates. The Hecke algebra $\mathcal{H}:=\mathcal{H}(q,Q_1,\ldots,Q_l)$ of 
 $G(l,1,n)$ over the ring $A:=\mathbb{C} [q^{\pm 1} ,Q_1,\ldots ,Q_l]$ is  the associative $A$-algebra   generated by $T_{0},\cdots ,T_{n-1}$ subject to the relations $(T_{0}-Q_1)\ldots (T_{0}-Q_l)=0$, $(T_{i}-q
)(T_{i}+1)=0$, for $1\leq i\leq n-1$ together with the following ``braid relations'':
\begin{gather*}
(T_{0}T_{1})^{2}=(T_{1}T_{0})^{2},\quad T_{i}T_{i+1}T_{i}=T_{i+1}T_{i}T_{i+1}%
\text{ }(1\leq i<n), \\
T_{i}T_{j}=T_{j}T_{i}\text{ }(j\geq i+2).
\end{gather*}
Set $K=\textrm{Frac} (A)$  then by  Tits' deformation theorem, the irreducible representations of $K\mathcal{H}:=K\otimes_A \mathcal{H}$ are naturally labeled by the set of $l$-partitions of $n$ (or simply multipartitions) that is, the set of $l$-tuples of partitions (written in decreasing order) of total sum $n$. We denote $\ulambda \vdash_l n$ if $\ulambda$ is a $l$-partition $\ulambda =(\lambda^1,\ldots,\lambda^l)$ of $n$.  
  In fact, there exists a certain set of $\mathcal{H}$-modules $V^{\ulambda}$ with  $\ulambda \vdash_l n$ such that:
$$\textrm{Irr} ( K\mathcal{H})=\{ K\otimes_A V^{\ulambda} \ |\ \ulambda \vdash_l n\}.$$
 If we consider a specialization $\theta : A \to \mathbb{C}$, the specialized algebra $\mathbb{C}_{\theta}\mathcal{H}:=\mathbb{C}\otimes_A \mathcal{H}$ is non semisimple in general. As a consequence the specialized $\mathbb{C}_{\theta}\mathcal{H}$-modules $\mathbb{C}_{\theta}V^{\ulambda}:= \mathbb{C} \otimes_{A} V^{\ulambda}$ are non semisimple in general. 
 
 In this case, a reduction theorem by Dipper and Mathas asserts that it is sufficient to consider the following case to understand  the representation theory of our algebra (in the sense that knowing the labelings, the dimensions and the characters of the simple modules in this case is sufficient to know them in general):
 $$\theta (q)=\eta\in \mathbb{C}^{*},\ \theta (Q_j)=\eta^{s_j}\  (j=1,\ldots,l),$$
 where $\eta$ is a primitive $e$-root of $1$ where $e\in \mathbb{Z}_{>1}$  and ${\bf s}:=(s_1,\ldots,s_l)\in \mathbb{Z}^l$.  

Now, it is known, since  the works of Lascoux, Leclerc, Thibon and Ariki, that a   part of the representation theory of Ariki-Koike algebras is controlled by  the quantum affine algebra in affine type $A$ and its action on the Fock space.  This approach allows a natural labeling of the simple modules by a 
 certain subset of $l$-partitions $\Phi_{{\bf s},e }(n)$ known as the ``Uglov $l$-partitions'': 
 $$\textrm{Irr} ( \mathbb{C}_{\theta_{\bf s}} \mathcal{H})=\{ D_{{\bf s},e}^{\umu} \ |\ \umu \in \Phi_{{\bf s},e} (n) \}.$$
 If we  consider a one dimensional $\mathcal{H}$-module $V^{\ulambda}$. Then there exists $\ulambda_{\theta_{\bf s}}\in \Phi_{{\bf s},e }(n)$  such that 
  $\mathbb{C}_{\theta_{\bf s}} V^{\ulambda}\simeq D_{{\bf s},e }^{\ulambda_{\theta_{\bf s}}}$. The goal of this note is to explicitly calculate  $\ulambda_{\theta_{\bf s}}$. 
   To do this,  after stating the main problem we are interested in, several preparatory results must and will be made, first around the parametrization $\Phi_{{\bf s},e} (n)$ and then around the strategy to identify our one dimensional representations. The fifth part recalls and translates several results of crystal graph theory which will be needed in the proof of our main results in the sixth section. The last section deals specifically with the case where $l=2$  (that is the Hecke algebra of type $B$) which is the first motivation of this work. This in particular solves  the problem settled  in \cite[\S 3.8]{ST}.

 \paragraph{Acknowledgements.}
 The author thanks Meinolf Geck for asking him this problem and for useful discussion around it.   This work is supported by 
 Agence National de la Recherche Projet ACORT ANR-12-JS01-0003.

\section{The main problem} 

We here introduce several  notations and state our main problem. 
 
 \begin{abs} We keep the notations of the introduction. A natural labeling  of the irreducible representations for  $K\mathcal{H}$ may be obtained using a natural labeling of the irreducible representations of $\mathbb{C}G(l,1,n)$ together with Tits' deformation Theorem. This is done as follows. We have:
 $$\textrm{Irr} (\mathbb{C}G(l,1,n))=\{ E^{\ulambda} \ |\ \ulambda\vdash_l n\}.$$
 (We refer to \cite[\S 5.1.3]{GJ} for the construction of the irreducible $\mathbb{C}G(l,1,n)$-modules.)
 Now  the specialization $\theta_{\textrm{can}}:A \to \mathbb{C}$ sending $Q_j$ to $\eta_l^{j-1}$ for $j=1,\ldots,l$  (where $\eta_l:=\textrm{exp}(2i\pi/l)$) and $q$ to $1$ induces a canonical bijection between the 
  sets $\textrm{Irr} (KG(l,1,n))$ and $\textrm{Irr} (K\mathcal{H})$ (see for example \cite[\S 8.1.6]{GP}.) We can thus denote:
 $$\textrm{Irr} (K\mathcal{H})=\{ KV^{\ulambda} \ |\ \ulambda\vdash_l n\}.$$
In our case (the algebra $\mathcal{H}$ is cellular in the sense of Graham and Lehrer) and at the level of characters, this bijection may be seen as follows. Take a simple  $K\mathcal{H}$-module $KV$  with character $\chi:K\mathcal{H}\to K$. Then it can be shown that 
there exists  a $\mathcal{H}$-module $V$ with character $\chi_A :\mathcal{H} \to A$  such that $\chi_A$ extends to $\chi$ and $K\otimes_A V=KV$. Applying the specialization $\theta$ to $\chi_A$ leads to a  trace function $\chi_\mathbb{C}: \mathbb{C}G(l,1,n)  \to \mathbb{C}$, which  is the character of a
 simple  $\mathbb{C}G(l,1,n)$-module $E^{\ulambda}$. $KV$ is thus naturally labeled with $\ulambda$ and we can thus denote $KV=KV^{\ulambda}$, which is obtained from the $\mathcal{H}$-module $V$ by extension of scalars.

In particular, let us consider the labeling of the one dimensional representations for $K \mathcal{H}$. They are obtained by extension of scalars of the one dimensional representations for $\mathcal{H}$ 
given as follows. 
  \begin{enumerate}
 \item For $j=1,\ldots, l$, the representation  $\rho : \mathcal{H} \to A$ such that $\rho (T_0)=Q_j$ and $\rho (T_i) = q$ (for $i=1,\ldots, n-1$) is labeled by the $l$-partition $\ulambda$ such that $\lambda^i=\emptyset$ if $i\neq j$ and $\lambda^j=(n)$.  For convenience, 
 we denote by $[n,j]$ such a multipartition. 
 
 \item For $j=1,\ldots, l$, the representation  $\rho : \mathcal{H} \to A$ such that $\rho (T_0)=Q_j$ and $\rho (T_i) = -1$  (for $i=1,\ldots, n-1$)  is labeled by the $l$-partition $\ulambda$ such that $\lambda^i=\emptyset$ if $i\neq j$ and $\lambda^j=(1^n)$ (which means that $1$ is repeated $n$ times.)  We denote by $[1^n,j]$ such a multiparitition.
 \end{enumerate}
 We will denote by $\Lambda (n)$ the subset of $l$-partitions consisting in these two types of $l$-partitions. \end{abs}
 
\begin{abs}\label{spe} Let ${\bf s}=(s_1,\ldots,s_l)\in\mathbb{Z}^l$. Assume that we have a specialization $\theta_{{\bf s}} :A \to \mathbb{C}$ such that $\theta_{{\bf s}}  (q)=\eta$, a primitive root of unity of order $e>1$ and 
 $\theta_{{\bf s}}  (Q_i)=\eta^{s_i}$ for $1\leq i\leq l$. 
  We denote by $T^{\theta_{{\bf s}} }_0$, $T^{\theta_{{\bf s}} }_1$, \ldots, $T^{\theta_{{\bf s}} }_{n-1}$ the standard generators of the specialized algebra $\mathbb{C}_{\theta_{{\bf s}} } \mathcal{H}$.  
  For $\ulambda \vdash_l n$, let us  consider  
  the  $\mathcal{H}$-module $V^{\ulambda}$. After specialization, this module, which is denoted by $\mathbb{C}_{\theta_{{\bf s}} } V^{\ulambda}$,  is non simple (nor semisimple) in general but we have  an associated composition serie. Let us denote by 
 $[\mathbb{C}_{\theta_{{\bf s}} }V^{\ulambda} : M]$ the multiplicity of $M\in \operatorname{Irr} (\mathbb{C}_{\theta_{{\bf s}} } \mathcal{H})$ in such a composition serie (this is well-defined by the Jordan-H\"older theorem). Then the matrix defined by:
$$\mathcal{D}_{\theta_{{\bf s}} } :=( [\mathbb{C}_{\theta_{{\bf s}} } V^{\ulambda} : M])_{\ulambda \vdash_l n,M\in \operatorname{Irr} (\mathbb{C}_{\theta_{\bf s}} \mathcal{H})}$$
controls a part of the representation theory of $\mathbb{C}_{\theta_{{\bf s}} }\mathcal{H}$.This is called the decomposition matrix with respect to the specialization $\theta_{\bf s}$. It  can be  used to parametrize the simple modules by certain subset of $l$-partitions as we briefly recall now. 

\end{abs}

\begin{abs} They are several ways of indexing the simple modules of the Ariki-Koike algebras after specializations. To do this, we can 
 use the theory of canonical basic sets developed by Geck and Rouquier and generalized by Gerber \cite{Ge1}. We here follow \cite[Ch.5, Ch. 6]{GJ}.   Consider another $l$-tuple of integers $(v_1,\ldots,v_l)$ such that for all $i<j$ then $0\leq v_j-v_i<e$  and define ${\bf m}=(m_1,\ldots,m_l)\in \mathbb{Q}^l$ such that for all $j=1,\ldots,l$, $m_j=s_j-v_j$. 
  Then one can define a pre-order $\ll_{\bf m}$ on the set of  
 $l$-partitions which depends on the choice of ${\bf m}$. We don't give the definition of this pre-order here, all we need to know is the following theorem (see 
 \cite[\S 6.7]{GJ}). 
 \end{abs}
  \begin{Th}\label{basic} Under the above hypotheses, 
there exists  a subset $\Phi_{{\bf s},e} (n)$ of the set of $l$-partitions of rank $n$ such that for all $M\in \operatorname{Irr} (\mathbb{C}_{\theta_{{\bf s}} } \mathcal{H})$, 
  \begin{enumerate}
  \item there exists $\ulambda_M\in \Phi_{{\bf s},e} (n)$  such that $[\mathbb{C}_{\theta_{\bf s}} V^{\ulambda_M} : M]=1$,
  \item for all $\mu\vdash_l n$, if $[\mathbb{C}_{\theta_{\bf s}} V^{\umu} : M]\neq 0$ then $\umu\ll_{\bf m} \ulambda_M$.
  \end{enumerate}
  The map $M\mapsto \ulambda_M$ is injective. As a consequence, if for all $M\in \operatorname{Irr} (\mathbb{C}_{\theta_{{\bf s}} } \mathcal{H})$ we denote  $D_{{\bf s},e}^{\ulambda_M}:=M$, we have:
$$\operatorname{Irr} (\mathbb{C}_{\theta_{{\bf s}} } \mathcal{H})=\{ D_{{\bf s},e}^\umu \ |\ \umu \in \Phi_{{\bf s},e} (n)  \}.$$
  \end{Th}
\begin{Rem}
This approach has been generalized in \cite{Ge1} where more general pre-orders $\ll_{\bf m}$ are considered. 

\end{Rem}

\begin{abs}\label{pb}
Our main problem is now the following. Take $\ulambda\in \Lambda (n)$ then, we want to calculate $\ulambda_{\theta_{\bf s}}\vdash_l n $ which  is characterized as follows:
\begin{enumerate}
\item we have $[\mathbb{C}_{\theta_{\bf s}} V^{\ulambda_{{\theta}_{\bf s}}} :  \mathbb{C}_{\theta_{\bf s}} V^{\ulambda}   ]=1$,
\item if we have $[\mathbb{C}_{\theta_{\bf s}} V^{\umu}:  \mathbb{C}_{\theta_{\bf s}} V^{\ulambda}   ]\neq 0$ then 
  $\umu\ll_{\bf m} \ulambda_{\theta_{\bf s}}$.

\end{enumerate}
Indeed, by the above theorem, we obtain $D_{{\bf s},e}^{\ulambda_{\theta_{{\bf s}} }} = \mathbb{C}_{\theta_{\bf s}} V^{\ulambda_{\theta_{\bf s}}} $. 
In particular we have $\ulambda_{\theta_{\bf s}} \in  \Phi_{{\bf s},e} (n)$. We denote by $\Lambda_{{\bf s},e} (n)$ the 
 set consisting of all the $l$-partitions  $\ulambda_{\theta_{\bf s}}$ with $\ulambda \in \Lambda (n)$. 

 Note that if $e=2$, for all $j=1,\ldots,l$, the representations $\mathbb{C}_{\theta_{\bf s}} V^{[1^n,j]}$
   and  $\mathbb{C}_{\theta_{\bf s}} V^{[n,j]}$ are isomorphic. As a consequence we must have $[1^n,j]_{{\theta_{\bf s}} }=[n,j]_{{\theta_{\bf s}} }$ in this case.

\end{abs}
\section{Characterization of   the sets $\Phi_{{\bf s},e} (n)$}

In this part, we explain how one can compute the sets $\Phi_{{\bf s},e} (n)$  we have defined above and study several properties around these subsets of $l$-partitions. 

  \begin{abs}\label{good}
  The subsets $\Phi_{{\bf s},e} (n)$ are in fact certain subsets of $l$-partitions 
which label the crystal of some irreducible highest weight modules in quantum affine type $A$. A complete survey on this subject can be found  in \cite[Ch. 6]{GJ}. We can give a purely combinatorial (but recursive) definition of $\Phi_{{\bf s},e} (n)$   as follows. Let $\ulambda$ be an $l$-partition. The nodes of $\ulambda$ are by definition the elements of   the Young diagram of $\ulambda$:
$$[\ulambda]:=\{ (a,b,c)  \ | \ a\geq 1,\ c\in \{1,\ldots,l\},\ 1\leq b\leq \lambda_a^c\} \subset \mathbb{Z}_{>0}\times 
  \mathbb{Z}_{>0} \times \{1,\ldots,l\}.$$
  The residue of  a node $\gamma=(a,b,c)$ of $\ulambda$ is the element $b-a+s_c+e\mathbb{Z}$  of $\mathbb{Z}/e\mathbb{Z}$.  
If $[\ulambda]=[\umu] \cup \{\gamma\}$ for some $l$-partition of $n-1$, we say that   $\gamma \in [\ulambda]$ is a removable node for $\ulambda$ and an addable node for $\umu$.  If the residue of $\gamma $ is $i\in \mathbb{Z}/e\mathbb{Z}$, then we say that $\gamma$ is an $i$-node. We now define a total order on the set of removable and addable $i$-nodes of $\ulambda$ depending on ${\bf s}$ and $e$. Let $\gamma=(a,b,c)$ and $\gamma'=(a',b',c')$ be such two $i$-nodes. Then we denote 
$$\gamma \prec_{{\bf s},e} \gamma' \iff \left\{ \begin{array}{ll}
\text{either}&b-a+s_c<b'-a'+s_{c'}, \\
\text{or}&b-a+s_c=b'-a'+s_{c'} \text{ and }c>c'.
\end{array}\right.$$
For $\ulambda$ an $l$-partition, we can consider its set of addable and removable $i$-nodes.  Let $w_{i}(\ulambda)$ be the word obtained first by writing the
addable and removable $i$-nodes of ${\boldsymbol{\lambda}}$ in {increasing}
order with respect to $\prec _{{\mathbf{s}},e}$,  
next by encoding each addable $i$-node by the letter $A$ and each removable $%
i$-node by the letter $R$.\ Write $\widetilde{w}_{i}(\ulambda)=A^{p}R^{q}$ for the
word derived from $w_{i}$ by deleting as many of the factors $RA$ as
possible. $w_{i}(\ulambda)$ is called the $i(\textrm{mod }e)$-word of $\ulambda$ and 
 $\widetilde{w}_{i}(\ulambda)$ the reduced $i(\textrm{mod }e)$-word of $\ulambda$ .

If $p>0,$ let $\gamma $ be the rightmost addable $i$-node in $%
\widetilde{w}_{i}$. The node $%
\gamma $ is called the {good addable $i$-node}. If $r>0$, the leftmost removable $i$-node in 
 $\widetilde{w}_{i}$ is called the {good removable $i$-node}.

Then we have $\ulambda \in  \Phi_{{\bf s},e} (n)$ if and only if there exist a sequence of elements $g_{{\bf s},e}(\ulambda)=(i_1,\ldots,i_n)\in (\mathbb{Z}/e\mathbb{Z})^n$ and, for each $k=1,\ldots,n$, an $l$-partition  $\ulambda[k]$ of $k-1$ such that 
 $\ulambda[1]=\emptyset$ and $\ulambda [n]=\ulambda$ such that for all $k=2,\ldots,n$, 
 $[\ulambda[k]]=[\ulambda[k-1]]\cup \{\gamma\}$ for a good addable $i$-node for $\ulambda [k-1]$. Note that $\ulambda$ is entirely and uniquely determined by the datum of $g_{{\bf s},e}(\ulambda)$.

  \end{abs}
  \begin{Exa}
  Take $(s_1,s_2)=(2,0)$ and $e=4$. We consider the $2$-partition $((4),(2.1))$, the associated  pair of Young diagrams  (where we add the residue of each node in the associated box) is given by:

  $$
\left(
\begin{array}{|c|c|c|c|}
  \hline
  \ 2 \ &\ 3\  &\ 0\ &\ 1    \\
  \hline
\end{array}\;,\;
\begin{array}{|c|c|}
  \hline
  \ 0\  &\ 1\      \\
  \cline{1-2}
 \ 3\   \\
 \cline{1-1}
\end{array}
\right)$$

Let us look at the removable good $1$-node, we have two removable $1$-node $(1,4,1)$ and $(1,2,2)$
 and one addable $1$-node $(2,1,1)$. Moreover we have:
 $$(1,2,2)\prec_{{\bf s},e} (2,1,1)\prec_{{\bf s},e}   (1,4,1),  $$
 which gives 
  $${w}_1 ( ((4),(2.1)))= RAR.$$
  We conclude that $(1,4,1)$ is the good removable $1$-node of $((4),(2.1))$. Continuing in this way, we see that $((4),(2.1))$ is in $\Phi_{{\bf s},e} (7)$ with $g_{{\bf s},e} ((4),(2.1))=(2,0,3,3,0,1,1)$.
\end{Exa}
  \begin{abs}\label{aff}
Let $\widehat{\mathfrak{S}}_l$ be the (extended) affine symmetric group. This is defined as follows. We denote by $P_l:=\mathbb{Z}^l$ the $\mathbb{Z}$-module with  standard basis $\{y_i\ |\ i=1,\ldots, l\}$. For $i=1,\ldots,l-1$, we denote by $\sigma_i$ the transposition $(i,i+1)$ of  $\mathfrak{S}_l$. Then 
$\widehat{\mathfrak{S}}_l$ can be seen as 
  the semi-direct product $P_l \rtimes \mathfrak{S}_l$ where the relations  are given by $\sigma_i y_j=y_j \sigma_i$ for $j\neq i,i+1$ and $\sigma_i y_i \sigma_i=y_{i+1}$ for $i=1,\ldots,l-1$ and $j=1,\ldots,l$.  
This group acts faithfully on $\mathbb{Z}^l$  by setting for any ${{{\bf s}}}=(s_{1},\ldots ,s_{l})\in 
\mathbb{Z}^{l}$: %
$$\begin{array}{rcll}
\sigma _{c}.{{{\bf s}}}&=&(s_{1},\ldots ,s_{c-1},s_{c+1},s_{c},s_{c+2},\ldots ,s_{l})&\text{for }c=1,\ldots,l-1 \text{ and }\\
y_i.{{{\bf s}}}&=&(s_{1},s_{2},\ldots,s_i+e,\ldots ,s_{l})&\text{for }i=1,\ldots,l
\end{array}$$
A fundamental domain for this action is given by
$$\left\{(s_1,\ldots  ,s_{l})\in \mathbb{Z}^l\ |\ 0\leq s_1 \leq \ldots  \leq s_l <e  \right\}.$$
  \end{abs}
  
  \begin{abs} 
   Let ${\sigma}\in \widehat{\mathfrak{S}}_l$
 and assume that ${\bf s}\in \mathbb{Z}^l$. Set ${\bf s}':=\sigma.{\bf s}$.  
Recall that we have defined  a first specialization $\theta_{{\bf s}} :A \to \mathbb{C}$   in \S\ref{spe}  and let us  consider another specialization $\theta_{\sigma.{\bf s}} :A \to \mathbb{C}$ such that  $\theta_{\sigma.{\bf s}} (q)=\eta$ and 
 $\theta_{\sigma.{\bf s}} (Q_i)=\eta^{s_i'}$ for $1\leq i\leq l$. Then  there is an isomorphism of $\mathbb{C}$-algebras:
 $$\begin{array}{rcl}
  \mathbb{C}_{\theta_{\sigma.{\bf s}}} \mathcal{H}&\to& \mathbb{C}_{\theta_{{\bf s}}} \mathcal{H}\\
  T^{\theta_{\sigma.{\bf s}}}_0 & \mapsto & T_0^{\theta_{{\bf s}}} \\
  T^{\theta_{\sigma.{\bf s}}}_i & \mapsto & T_i^{\theta_{{\bf s}}} (i=1,\ldots,n-1).\\  
  \end{array}
  $$
  This induces an exact functor $\mathcal{F}$  from the category of finite dimensional $\mathbb{C}_{\theta_{{\bf s}}} \mathcal{H}$-modules to the category of finite dimensional 
  $\mathbb{C}_{\theta_{\sigma.{\bf s}}} \mathcal{H}$-modules. As a consequence, we obtain a bijection: 
 $$\Psi_{e,{\bf s},{\bf s}'} :  \Phi_{{\bf s},e} (n) \to  \Phi_{{\bf s}',e} (n),$$
 which is defined as follows. Let $\ulambda \in \Phi_{{\bf s},e} (n)$ then there exists $\Psi_{e,{\bf s},{\bf s}'} (\ulambda) \in \Phi_{{\bf s}',e} (n)$
  such that $\mathcal{F} (D_{{\bf s},e}^{\ulambda})=D_{{\bf s}',e}^{\Psi_{e,{\bf s},{\bf s}'} (\ulambda)}$. 

  Take $\ulambda \vdash_l n$ and consider the irreducible  $\mathbb{C}_{\theta_{{\bf s}}} \mathcal{H}$-module $\mathbb{C}_{\theta_{{\bf s}}} V^{\ulambda}$. Applying the functor $\mathcal{F}$ gives a  $\mathbb{C}_{\theta_{\sigma.{\bf s}}} \mathcal{H}$-module.
  We set $\sigma=\sigma_0.y$ with $\sigma_0\in \mathfrak{S}_l$ and $y\in P_l$.
   It is important to note that  we have $\theta_{\sigma.{\bf s}}=\theta_{\sigma_0 .{\bf s}}$ because $\eta$ is an $e$-root of $1$. 
   By the construction of the irreducible representations (see \cite[Ch. 13]{A}),   it is easily checked that this   $\mathbb{C}_{\theta_{\sigma.{\bf s}}} \mathcal{H}$-module is isomorphic to $\mathbb{C}_{\theta_{\sigma.{\bf s}}} V^{\sigma_0. \ulambda}$ where $\mathfrak{S}_l$ acts naturally on the set of $l$-partitions by permutation (that is $\sigma_0.\ulambda=(\lambda^{\sigma_0^{-1} (1)},\ldots, \lambda^{\sigma_0^{-1} (l)})$.)
  \begin{Exa}
  In particular, if  $\ulambda \in \Lambda (n)$, then there exists $i\in \{1,\ldots,l\}$ such that the $\mathbb{C}_{\theta_{{\bf s}}} \mathcal{H}$-module $\mathbb{C}_{\theta_{{\bf s}}} V^{\ulambda}$   sends  $T^{\theta_{{\bf s}}}_0$  to $\theta_{\bf s} (Q_i)$ and $T^{\theta_{{\bf s}}}_i$ to $\theta_{{\bf s}} (q)$ or $-1$. Hence 
  $\mathcal{F}(\mathbb{C}_{\theta_{{\bf s}}} V^{\ulambda})$ is the $\mathbb{C}_{\theta_{\sigma.{\bf s}}} \mathcal{H}$-module sending 
   $T_0^{\sigma.\theta_{{\bf s}}}$ to $\theta_{\sigma.{\bf s}} (Q_{\sigma_0  (i)})$. We obtain that
  $\mathcal{F}(\mathbb{C}_{\theta_{{\bf s}}} V^{\ulambda})=\mathbb{C}_{\theta_{\sigma.{\bf s}}} V^{\sigma_0. \ulambda}$.
  \end{Exa}
  As a consequence, if we restrict the bijection 
  $\Psi_{e,{\bf s},{\bf s}'}$ to  $\Lambda_{{\bf s},e} (n)$, we obtain a bijection which is, by a slight abuse of notation,  noted identically:
$$\begin{array}{rccl}
\Psi_{e,{\bf s},{\bf s}'} :  &{\Lambda}_{{\bf s},e} (n)& \to&  {\Lambda}_{{\bf s}',e} (n)\\
&\ulambda_{\theta_{\bf s}} & \mapsto & {(\sigma_0.\ulambda)}_{\theta_{{\bf s}'}}.
\end{array}
$$
 Note that we have for all  $M\in \operatorname{Irr} (\mathbb{C}_{\theta_{{\bf s}}} \mathcal{H})$ and  for all $\umu \vdash_l n$:
 \begin{eqnarray}\label{eq1}
 [\mathbb{C}_{\theta_{{\bf s}}} V^{\umu} : M]=[\mathbb{C}_{\theta_{\sigma.{\bf s}}} V^{\sigma_0. \umu} : \mathcal{F}(M)].
\end{eqnarray}
\end{abs}

\begin{abs} 
 On the other hand, as explained in \cite[\S 6.2.17]{GJ}, the crystal graph theory allows to construct a combinatorial bijection between the two sets. This is given as follows: let  $\ulambda \in \Phi_{{\bf s},e} (n)$ then  there exists a unique $\umu \in \Phi_{{\bf s}',e} (n)$
  such that $g_{{\bf s},e} (\ulambda)= g_{{\bf s}',e} (\umu)$, we set $\chi_{e,{\bf s},{\bf s}' }(\ulambda):=\umu$. This defines
a bijection 
  $$\chi_{e,{\bf s},{\bf s}'}  :  \Phi_{{\bf s},e} (n) \to  \Phi_{{\bf s}',e} (n).$$
The following result asserts that the two bijections we have just constructed actually coincide. In the proof, we will freely use several results from \cite{GJ}.  
\end{abs}

 \begin{Prop} Let ${\bf s}\in \mathbb{Z}^l$, $\sigma\in \widehat{\mathfrak{S}}_l$ and ${\bf s}':=\sigma.{\bf s}$. 
 For all $\ulambda\in   \Phi_{{\bf s},e} (n) $, we have $\chi_{e,{\bf s},{\bf s}'}  (\ulambda)=\Psi_{e,{\bf s},{\bf s}'}  (\ulambda)$. 
 
 \end{Prop}
 \begin{proof}
 The proof is based on Ariki's Theorem relying the decomposition matrix of Ariki-Koike algebras with the canonical bases for Fock spaces. 
  By definition, the Fock space $\mathfrak{F}_{\bf s}$ is the $\mathbb{Q} (v)$-vector space (where $v$ is an indeterminate) with basis given by the symbols $|\ulambda, {\bf s} \rangle$ with $\ulambda \vdash_l n$ and $n\in \mathbb{Z}_{\geq 0}$. One can define an action of the quantum affine algebra of type $A$ on this space. The submodule generated by $|\uemptyset, {\bf s} \rangle$ is then an irreducible highest weight module (where $\uemptyset=(\emptyset,\ldots,\emptyset)$) and we have an associated canonical basis which is  indexed by the set $\Phi_{{\bf s},e}=\cup_{n\in \mathbb{Z}_{\geq 0}} \Phi_{{\bf s},e}(n)$:
  $$\{ G(\ulambda,{\bf s})\ |\ \ulambda \in \Phi_{{\bf s},e}\}.$$
Fix $n\in \mathbb{Z}_{>0}$ then for all  $\ulambda \in \Phi_{{\bf s},e}(n)$ and $\unu \vdash_l n$, there exist $d^{\bf s}_{\unu,\ulambda} (v) \in \mathbb{N}[v]$ such that:
  $$G(\ulambda,{\bf s})=\sum_{\unu\vdash_l n} d^{\bf s}_{\unu,\ulambda} (v) |\unu,{\bf s} \rangle.$$
  Now we fix $\ulambda \in \Phi_{{\bf s},e} (n)$.   By  Ariki's Theorem (see \cite[Thm 6.2.21]{GJ}),   there exists $M\in \textrm{Irr} (\mathbb{C}_{\theta_{\bf s}} \mathcal{H})$ such that  for all $\unu\vdash_l n$,
  $$d^{\bf s}_{\unu,\ulambda} (1)=[\mathbb{C}_{{\theta_{\bf s}}} V^{\unu}:M].$$
  By \cite[Thm 6.6.12]{GJ}  (see also  Thm \ref{basic}),   we have $M=D^{\ulambda}_{{\bf s },e}$. 
  On the other hand,  the submodule generated by $|\uemptyset, {\bf s}' \rangle$ in the Fock space $\mathfrak{F}_{\bf s '}$ admits also a canonical basis 
  $$\{ G(\umu,{\bf s}')\ |\ \umu \in \Phi_{{\bf s}',e}\}.$$
 Set $\ulambda':=\chi_{e,{\bf s},{\bf s}'} (\ulambda)   \in \Phi_{{\bf s}',e}(n)$.   By  Ariki's Theorem again,   there exists $M'\in \textrm{Irr} (\mathbb{C}_{\theta_{{\bf s}'}} \mathcal{H})$ such that  for all $\unu\vdash_l n$,
  $$d^{{\bf s}'}_{\unu,\ulambda'} (1)=[\mathbb{C}_{{\theta_{{\bf s}'}}} V^{\unu}:M'].$$
  Again by \cite[Thm 6.6.12]{GJ}, we have $M'=D^{\ulambda'}_{{\bf s }',e}$. Now by \cite[\S 6.4.10]{GJ},  we have 
    $$d^{\bf s}_{\unu,\ulambda} (1)= d^{{\bf s}'}_{\sigma_0.\unu, \ulambda'} (1),$$
  and this leads to:
  $$[\mathbb{C}_{{\theta_{\bf s}}} V^{\unu}:   D^{\ulambda}_{{\bf s },e}]=
  [\mathbb{C}_{{\theta_{\sigma.{\bf s}}}} V^{\sigma_0.\unu}:   D^{\ulambda'}_{{\bf s },e}],$$
  for all $\unu\vdash_l n$.  Combining this with (\ref{eq1})  and taking into account that the above decomposition numbers determine uniquely the simple modules,   we obtain that  $D^{\ulambda'}_{{\bf s },e}=\mathcal{F} (D^{\ulambda}_{{\bf s },e})$ which is exactly what we wanted.



  \end{proof}
  \begin{abs}\label{strate} We end this section by giving 
  one  strategy to solve our problem \S \ref{pb}:
  \begin{enumerate}
 \item We find  all the elements $\ulambda_{\theta_{\bf s}}$ for all  $\ulambda \in \Lambda (n)$ and for all specialisation $\theta_{\bf s}$  
  where  ${\bf s} \in \mathcal{A}_{e,l}$, a domain contained in a fundamental domain with respect to the action of $\widehat{\mathfrak{S}}_l$
  on $\mathbb{Z}^l$. 
  \item Let ${\bf s}\in \mathbb{Z}^l$ and and let $\ulambda \in  \Lambda (n)$ then there exists $\sigma \in \widehat{\mathfrak{S}}_l$
   such that $\sigma.{\bf s}'= {\bf s}$ with ${\bf s}' \in \mathcal{A}_{e,l}$. 
    Then consider $\sigma_0^{-1} .\ulambda \in \Lambda (n)$. We can compute $(\sigma_0^{-1} \ulambda)_{{\theta}_{{\bf s}'}}$ by $(1)$.
    \item We have $(\sigma_0^{-1} \ulambda)_{\theta_{{\bf s}'}}\in \Phi_{{\bf s}',e} (n)$ and we can consider the sequence $g_{{\bf s}',e} ((\sigma_0^{-1} \ulambda)_{{\theta}_{{\bf s}'}})$. 
    \item By the discussion above, $\ulambda_{\theta_{\bf s}}$ is then uniquely determined by its sequence 
$g_{{{\bf s}},e} ( \ulambda_{\theta_{\bf s}})=g_{{{\bf s}'},e} ((\sigma_0^{-1} \ulambda)_{{\theta}_{{\bf s}'}})$.
\end{enumerate}

  \end{abs}

%

 \section{One dimensional representations and Mullineux involution}

Now our main problem can be rephrased as follows. Let $\ulambda \in \Lambda (n)$ then there exists $\ulambda_{\theta} \in \Phi_{{\bf s},e} (n)$ such that 
$$\mathbb{C}_{\theta_{\bf s}} V^{\ulambda}\simeq D_{{\bf s},e}^{\ulambda_{\theta_{\bf s}}}.$$
We want  to find explicitly $\ulambda_{\theta_{\bf s}}$ in terms of $\ulambda$ for all ${\bf s}\in \mathbb{Z}^l$ and $e\in \mathbb{Z}_{>1}$ . To do this, we will proceed in several steps. For the first step, the aim is to study the case where ${\bf s} \in \mathcal{A}_{e,l}$ where:
$$\mathcal{A}_{e,l}=\{ {\bf s}=(s_1,\ldots,s_l)\in \mathbb{Z}^l\ |\ \forall i<j,\ 0\leq s_j-s_i<e\}.$$
Note that the fundamental domain in \S \ref{aff} is contained in this set. 
We assume that ${\bf s} \in \mathcal{A}_{e,l}$.  In this  case, the set $\Phi_{{\bf s},e} (n)$ can be explicitly described without any references to the theory of crystal graphs. This set is given by the so called FLOTW $l$-partitions (see \cite[Def. 5.7.8]{GJ}), we recall the definition of them hereafter. 
\begin{abs}\label{flotw}
Recall that we assume that  ${\bf s} \in \mathcal{A}_{e,l}$.  We have $\ulambda=(\lambda^1,\ldots,\lambda^l)\in \Phi_{{\bf s},e} (n)$ 
 if and only if:
 \begin{enumerate}
 \item For all $j=1,\ldots,l-1$ and $i\in \mathbb{Z}_{>0}$, we have:
 $$\lambda_i^j\geq \lambda_{i+s_{j+1}-s_j}^{j+1}.$$
 \item For all $i\in \mathbb{Z}_{>0}$, we have:
 $$\lambda_i^{l}\geq \lambda_{i+e+s_{1}-s_l}^{1}.$$ 
 \item For all $k\in \mathbb{Z}_{>0}$, the set 
 $$\{ \lambda_i^j-i+s_j+e\mathbb{Z}\ |\ i\in \mathbb{Z}_{>0},\ \lambda_i^j=k, j=1,\ldots,l\},$$
 is a proper subset of $\mathbb{Z}/e\mathbb{Z}$. 
 
 \end{enumerate}
(In the above definition, the partitions are considered with an infinite number of empty parts)

\end{abs}

\begin{abs}\label{disc1}
First, note that $1$-dimensional representations may induced isomorphic modules after specialization. Indeed, let us consider $i_1=1<i_2, \ldots<i_k$ in $\{1,\ldots,l\}$ such that 
$s_{1}=\ldots =s_{i_2-1}$, $s_{i_2-1}<s_{i_2}$,   $s_{i_2}=\ldots =s_{i_3-1}$,  $s_{i_3-1}<s_{i_3}$,
 \ldots $s_{i_{k-1}-1}<s_{i_k}$,   $s_{i_k}=\ldots =s_{l}$, then we have for all $j=1,\ldots,k$:
 \begin{equation}\label{s1}
 \mathbb{C}_{\theta_{\bf s}}V^{[n,i_j]}\simeq \mathbb{C}_{\theta_{\bf s}}V^{[n,i_j+1]} \simeq \ldots \simeq \mathbb{C}_{\theta_{\bf s}}V^{[n,i_{j+1}-1]}\textrm{ and }
 \mathbb{C}_{\theta_{\bf s}}V^{[1^n,i_j]}\simeq \mathbb{C}_{\theta_{\bf s}}V^{[1^n,i_j+1]} \simeq  \ldots \simeq \mathbb{C}_{\theta_{\bf s}}V^{[1^n,i_{j+1}-1]},
 \end{equation}
 where we set $i_{k+1}:=l+1$. 
 Our main problem can now be easily solved in ``half'' of the cases.
\end{abs}
\begin{Prop}\label{triv}
For all $j=1,\ldots,k$, we have:
$$[n,i_j]_{\theta_{\bf s}}=[n,i_j+1]_{\theta_{\bf s}}=\ldots=[n,i_{j+1}-1]_{\theta_{\bf s}}=[n,i_j].$$ 
\end{Prop}
\begin{proof}
From (\ref{s1}), we immediately get that $[n,i_j]_{\theta_{\bf s}}=[n,i_j+1]_{\theta_{\bf s}}=\ldots=[n,i_{j+1}-1]_{\theta_{\bf s}}$. Now from the definition above, it is direct to see that $[n,i_j]\in \Phi_{{\bf s},e} (n)$ . As in addition $\mathbb{C}_{\theta_{\bf s}}V^{[n,i_j]}$ is simple we have 
$\mathbb{C}_{\theta_{\bf s}} V^{[n,i_j]}\simeq M$ for a simple module $M\simeq D_{{\bf s},e}^{[n,i_j]}  $ which is exactly what we wanted. 
\end{proof}
\begin{abs}\label{flows}
Now the FLOTW $l$-partitions of the above types  are characterized by the sequence $g_{{\bf s},e} ([n,i_j])$ of ``good nodes''. This sequence is easy to have in this case, it is indeed given by the following:
$$g_{{\bf s},e} ([n,i_j])=(s_{i_j}+e\mathbb{Z},s_{i_j}+1+e\mathbb{Z},\ldots,s_{i_j}+n-1+e\mathbb{Z}).$$

\end{abs}

\begin{abs}
For the one dimensional representation indexed by $[1^n,i_j]$ ($j=1,\ldots,k$), the problem is more difficult  as such $l$-partitions are  not FLOTW $l$-partitions in general. 
 To solve it, let us consider another specialization $\widetilde{\theta}_{\bf s}: A \to \mathbb{C}$ such that 
 $$\widetilde{\theta}_{\bf s} (q)=\xi,\ \widetilde{\theta}_{\bf s} (Q_j)=  \xi^{-s_{l+1-j}},\ j=1,\ldots,l,$$
 where $\xi:=\eta^{-1}$. 
We denote by ${T}^{\widetilde{\theta}_{\bf s}}_0,\widetilde{T}^{\widetilde{\theta}_{\bf s}}_1, \ldots, \widetilde{T}^{\widetilde{\theta}_{\bf s}}_{n-1}$ the associated standard generators.   
   Note that the relations in $\mathbb{C}_{\widetilde{\theta}_{\bf s}}{\mathcal{H}}$ are given as follows:
   $$({T}^{\widetilde{\theta}_{\bf s}}_{0}-\xi^{-s_l})\ldots ({T}^{\widetilde{\theta}_{\bf s}}_{0}-\xi^{-s_1})=0,\  ({T}^{\widetilde{\theta}_{\bf s}}_{i}-\xi
)(\widetilde{T}^{\widetilde{\theta}_{\bf s}}_{i}+1)=0,\ i=1,\ldots,n-1.$$ 
 As $-{\bf s}:=(-s_l,\ldots,-s_1)\in \mathbb{Z}^l$ and $\xi$ is still a primitive $e$-root of $1$,  we can use the results of Theorem \ref{basic}.  In particular we have a parametrization of the simple modules by a certain set of  $l$-partitions 
   $\Phi_{-{\bf s},e} (n)$ which can be recursively described as in \S \ref{good}. 
   $$\operatorname{Irr} (\mathbb{C}_{\widetilde{\theta}_{\bf s}}\mathcal{H})=\{ {D}_{-{\bf s},e}^\umu \ |\ \umu \in \Phi_{-{\bf s},e} (n)  \}.$$

  Following an idea of M. Fayers \cite{Ma}, we consider the isomorphism of $\mathbb{C}$-algebra sending ${T}^{\widetilde{\theta}_{\bf s}}_0$ to $T^{{\theta}_{\bf s}}_0$ and 
${T}^{\widetilde{\theta}_{\bf s}}_i$ to $-\xi T^{{\theta}_{\bf s}}_i$ for $i=1,\ldots,n-1$. This, in turn, induces a functor $F$ from the category of 
 $\mathbb{C}_{\theta_{\bf s}}\mathcal{H}$-modules to the category of  $\mathbb{C}_{\widetilde{\theta}_{\bf s}}{\mathcal{H}}$-modules. 
    For all $\ulambda \in \Phi_{{\bf s},e} (n)$, there exists a bijection 
    $$m_{{\bf s},e}:   \Phi_{{\bf s},e} (n)\to \Phi_{-{\bf s},e} (n),$$
 such that:
    $$F(D_{{\bf s},e}^{\ulambda})={D}_{-{\bf s},e}^{m(\ulambda)}.$$
This map satisfies $m_{{\bf s},e} \circ m_{-{\bf s},e}=\textrm{Id}_{\Phi_{-{\bf s},e} (n)  }$  and $m_{-{\bf s},e} \circ m_{{\bf s},e}=\textrm{Id}_{\Phi_{{\bf s},e} (n)  }$
         and can be seen as a generalization of the Mullineux involution. Indeed, in the case $l=1$, this involution 
         is simply the usual Mullineux involution which we denote  by $m_e$.
         
\end{abs}
\begin{Prop}\label{mull}
For $j=1,\ldots,k$, we have $[1^n,i_j]_{\theta_{{\bf s}}}=m_{-{\bf s},e}([n,l+2-i_{j+1}])$ where $i_{k+1}:=l+1$. 

\end{Prop}
\begin{proof}
By definition, we have  $\mathbb{C}_{\theta_{{\bf s}}} V^{[1^n,i_j]}\simeq D_{{\bf s},e}^{[1^n,i_j]_{\theta_{{\bf s}}}}$ and we can apply the functor $F$ to obtain 
$F(\mathbb{C}_{\theta_{{\bf s}}} V^{[1^n,i_j]})\simeq F(D_{{\bf s},e}^{[1^n,i_j]_{\theta_{{\bf s}}}})$.  
On the one hand, by definition, we have $F(D_{{\bf s},e}^{[1^n,i_j]_{{\theta}_{\bf s}}})={D}_{-{\bf s},e}^{m_{{\bf s},e} ([1^n,i_j]_{{\theta}_{\bf s}})}$ and on the other hand the representation associated to the 
 $\mathbb{C}_{\widetilde{\theta}_{\bf s}} \mathcal{H}$-module $F(\mathbb{C}_{\theta_{{\bf s}}} V^{[1^n,i_j])})$ is given by
 $$
 \widetilde{T}_0^{\widetilde{\theta}_{\bf s}}\mapsto T^{\theta_{\bf s}}_0 \mapsto \eta^{s_{i_j}}=\xi^{-s_{i_j}},\ \widetilde{T}^{\widetilde{\theta}_{\bf s}}_i \mapsto -\eta^{-1} T_i^{\theta_{\bf s}} \mapsto \eta^{-1}=\xi.
$$
this is the module obtained by  specialization through $\widetilde{\theta}_{\bf s}$ of the $\mathcal{H}$-module $V^{[n,l+1-i_j]}$ sending ${T}_0$ to ${q}$ 
 and ${T}_i$ to ${Q}_{l+1-i_j}$  for $i=1,\ldots,n-1$. This is thus $\mathbb{C}_{\widetilde{\theta}_{\bf s}} V^{[n,l+1-i_j]}$. 
  Now, note that we have $-{\bf s}\in \mathcal{A}_{e,l}$. Hence we can use Proposition \ref{triv} which shows that:
  $$\mathbb{C}_{\widetilde{\theta}_{\bf s}}V^{[n,l+1-i_j]}={D}_{-{\bf s},e}^{[n,l+2-i_{j+1}]}.$$
As a conclusion we obtain: 
$$m_{{\bf s},e} ([1^n,i_j]_{\theta_{\bf s}})={[n,l+2-i_{j+1}]}.$$
Applying then $m_{-{\bf s},e}$ leads to the desired result. 
\end{proof}
\begin{abs}\label{exp}
By the result of Fayers, the image $m_{-{\bf s},e}(\ulambda)$ for $\ulambda \in \Phi_{{\bf s},e}(n)$ may be recursively computed. Indeed let us consider the sequence $g_{-{\bf s },e} (\ulambda)$ and  assume that we have:
$$g_{{\bf s },e} (\ulambda)=(k_1+e\mathbb{Z},\ldots, k_n+e\mathbb{Z}).$$
Then $m_{-{\bf s},e}(\ulambda)$  satisfies:
$$g_{{\bf s },e} (m_{-{\bf s},e}(\ulambda))=(-k_1+e\mathbb{Z},\ldots, -k_n+e\mathbb{Z}).$$
By  \S \ref{good}, this uniquely characterized  $m_{-{\bf s},e}(\ulambda)$. 

\end{abs}

\begin{Exa}\label{typA}
Let us study the case of type $A$ that is when $l=1$. In this case, note that $T_0$ is just $Q_1$ so we only have to consider  $\theta : A \to \mathbb{C}$ a specialization such that $\theta (q)=\eta$, a primitive $e$-root of $1$.  We have two $1$-dimensional representations $(n)_{\theta}$ and $(1^n)_{\theta}$ for the specialized Hecke algebra.

By Proposition \ref{triv}, we have $(n)_{\theta_{\bf s}}=(n)$ and  by Proposition \ref{mull}, we have $(1^n)_{\theta_{\bf s}}=m_e ((n))$ where $m_e$ is the usual Mullineux involution. The explicit form of this partition is known (by \cite[Thm 6.22]{James}) and can,  for example, easily been computed following \S \ref{exp}. Take here for example ${\bf s}=(0)$ (it does in fact not depend on the choice of ${\bf s}$ in this case)
$$g_{(0),e} ( (n) )=(e\mathbb{Z},1+e\mathbb{Z},\ldots, n-1+e\mathbb{Z}).$$
Then $m_{e}((n))$  is the unique element such that 
$$g_{(0),e} (m_{e}((n)))=(e\mathbb{Z},-1+e\mathbb{Z}, \ldots, -n+1+e\mathbb{Z}).$$
It is not difficult to see that we obtain the partition $((q+1)^r q^{e-1-r})$ 
where $(q,r)\in \mathbb{Z}_{>0}\times \{0,1,\ldots,e-2\}$ is uniquely defined by the euclidean division $n=q (e-1)+r$.

\end{Exa}

\begin{Exa}
 Assume now that $e=2$. Then we remark that we have for all $j=1,\ldots,k$:
 $$g_{{\bf s},e}([n,i_j])=g_{{\bf s},e}([1^n,i_j]),$$
 and more generally 
$$g_{{\bf s },e} (\ulambda)=g_{{\bf s },e} (m_{-{\bf s},e}(\ulambda)).$$
We thus have $[n,i_j]_{{\theta}_{\bf s}}=[1^n,i_j]_{{\theta}_{\bf s}}$ 
This is consistent with our remark in  \S \ref{pb}. 
\end{Exa}

  \section{Crystal isomorphisms}\label{crystal}
Let ${\bf s}\in \mathbb{Z}^l$ and ${\sigma}\in \widehat{\mathfrak{S}}_l$. Set ${\bf s}':=\sigma.{\bf s}=(s_1',\ldots,s_l')$.   
The aim of this part is to explain how the bijections   
    $$\chi_{e,{\bf s},{\bf s}'}  :  \Phi_{{\bf s},e} (n) \to  \Phi_{{\bf s}',e} (n),$$
  can be combinatorially computed without any reference to sequences of ``good nodes''.  This result essentially comes from \cite{JL} but we need to translate it a little bit for our purpose. This will be helpful in the following. 
  
 \begin{abs}   Recall \S \ref{aff} , if we set $\tau:=y_l \sigma_{l-1}\ldots\sigma_1$, we see that 
    $\widehat{\mathfrak{S}}_l$ is generated by $\tau$ and $\sigma_i$ for $i=1,\ldots,l-1$. In addition, the action of  $\widehat{\mathfrak{S}}_l$
   on $\mathbb{Z}^l$ gives that 
   $$\tau.{\bf s}=(s_2,\ldots,s_{l},s_1+e).$$
  To compute $\chi_{e,{\bf s},{\bf s}'}$ explicitly it suffices to consider the cases where $\sigma=\tau$ and $\sigma=\sigma_i$ for $i=1,\ldots,l$.  This is what we will do in the next subsections.
  \end{abs}
    \begin{abs}\label{algo0}
First,  we describe the case where $\sigma=\tau$. In this case, the bijection 
  $$\chi_{e,{\bf s},{\bf s}'}  :  \Phi_{{\bf s},e} (n) \to  \Phi_{{\bf s}',e} (n),$$
is simply given as follows, for all $\ulambda\in \Phi_{{\bf s},e} (n)$, we have:
  $$\chi_{e,{\bf s},{\bf s}'} (\ulambda)=(\lambda^2,\ldots,\lambda^l,\lambda^1).$$
We refer to \cite[Prop. 5.2.1]{JL} for a proof. 

  \end{abs}

  \begin{abs}\label{algo}
  Let us now describe the case  $\sigma=\sigma_i$ for $i=1,\ldots,l-1$. The proof follows from  the proof of \cite[Prop. 5.2.1]{JL}.  If $\ulambda=(\lambda^1,\ldots,\lambda^l)\in \Phi_{{\bf s},e} (n)$ we have 
      $$\chi_{e,{\bf s},\sigma_i.{\bf s}} (\ulambda)=(\lambda^1,\ldots,\lambda^{i-1},\widetilde{\lambda}^{i+1},\widetilde{\lambda}^{i},\lambda^{i+2},\ldots,\lambda^l),$$
 where we will now describe how one can obtain   $(\widetilde{\lambda}^{i+1},\widetilde{\lambda}^{i})$ from $(\lambda^i,\lambda^{i+1})$.
  This has been  originally given by modifying a combinatorial object associated to the $l$-partition called the  charged symbol.  Here we translate the procedure 
 in terms of Young tableaux. 

We in addition assume that  $s_i>s_{i+1}$. We only need this case in the following (but the result can be adapted to the case $s_i<s_{i+1}$) 
  We assume that $\lambda^i=(\lambda^i_1,\ldots,\lambda^i_{r_i})$ with $r_i\in \mathbb{Z}_{>0}$ and 
  $\lambda^{i+1}=(\lambda^{i+1}_1,\ldots,\lambda^{i+1}_{r_{i+1}})$ with $r_{i+1}\in \mathbb{Z}_{>0}$. We then set 
  $$r:=\text{Max} (r_{i+1}+s_i-s_{i+1},r_i).$$
  Now adding as many zero part as necessary we slightly abuse notation by assuming  that 
 $\lambda^i=(\lambda^i_1,\ldots,\lambda^1_{r})$  and 
  $\lambda^{i+1}=(\lambda^{i+1}_1,\ldots,\lambda^{i+1}_{r -s_i+s_{i+1}})$.
   Moreover, for any part $\mu_a^c$ of an $l$-partition $\umu=(\mu^1,\ldots,\mu^l)$, we set $\beta (\mu_a^c)=\mu_a^c-a+s_c$ (this is called a $\beta$-number). This corresponds to the 
    content of the right most node of $\lambda_a^c$ if it is non zero.

 We will now consider the parts of $\lambda^{i+1}$ indexed decreasingly from $r-s_i+s_{i+1}$ to $1$ and construct the new partitions $\widetilde{\lambda}^i$ and 
 $\widetilde{\lambda}^{i+1}$ 
  step by step:
 \begin{itemize}
 \item Let us consider the part $\lambda^{i+1}_{r-s_i+s_{i+1}}$, we define $a_1\in \{1,\ldots,r\}$ such that:
 $$\beta (\lambda^i_{a_1})=\textrm{Min} (\beta (\lambda^i_{a}),\ a\in \{1,\ldots,r\},\ \beta (\lambda^i_a)\geq \beta (\lambda^{i+1}_{r-s_i+s_{i+1}})).$$
 It can be shown from the fact that $\ulambda$ is in $\Phi_{{\bf s},e} (n)$ that this is well and uniquely defined. Then we set 
 
  $$\widetilde{\lambda}^{i}_k=\left\{
 \begin{array}{ccc}
0& \text{if }k>r-s_i+s_{i+1}, \\
{\lambda}^{i+1}_{r-s_i+s_{i+1}}+\beta (\lambda^i_{a_1})-\beta (\lambda^{i+1}_{r-s_i+s_{i+1}})& \text{if }r-s_i+s_{i+1}.
 \end{array}\right.$$
 
 $$\widetilde{\lambda}^{i+1}_k=\left\{
 \begin{array}{ccc}
 \lambda^i_k& \text{if }k>a_1,\\
 \lambda^i_{a_1}-\beta (\lambda^i_{a_1})+\beta (\lambda^{i+1}_{r-s_i+s_{i+1}})& \text{if }k=a_1.
 \end{array}\right.$$
 \item For $\lambda^{i+1}_{r-s_i+s_{i+1}-1}$, we define $a_2\in \{1,\ldots,r\}$ such that:
 $$\beta (\lambda^i_{a_2})=\textrm{Min} (\beta (\lambda^i_{a}),\ a\in \{1,\ldots,r\},\ \beta (\lambda^i_a)\geq \beta (\lambda^{i+1}_{r-s_i+s_{i+1}-1})).$$
 It can be shown from the fact that $\ulambda$ is in $\Phi_{{\bf s},e} (n)$  that this is well and uniquely defined and we have $a_2<a_1$. Then we set 
 $$\widetilde{\lambda}^i_{r-s_i+s_{i+1}-1}={\lambda}^2_{r-s_i+s_{i+1}-1}+\beta (\lambda^i_{a_2})-\beta (\lambda^{i+1}_{r-s_i+s_{i+1}-1}),$$
 $$\widetilde{\lambda}^{i+1}_k=\left\{
 \begin{array}{ccc}
 \lambda^i_k& \text{if }a_1>k>a_2,\\
 \lambda^i_{a_2}-\beta (\lambda^i_{a_2})+\beta (\lambda^{i+1}_{r-s_i+s_{i+1}-1})& \text{if }k=a_2.
 \end{array}\right.$$
 \item We continue this process until we reach $\lambda^{i+1}_1$. We thus have defined $a_j\in \{1,\ldots,r\}$ for $j=1,\ldots,r-s_i+s_{i+1}$ on the one hand and the parts $\widetilde{\lambda}^i_k$ for $k\geq 1$ and 
 $\widetilde{\lambda}^{i+1}_k$ for $k\geq a_{r-s_i+s_{i+1}}$ on the other hand. We in addition set  $\widetilde{\lambda}^{i+1}_k=\lambda^i_k$ for $1\geq k> a_{r-s_i+s_{i+1}}$. This gives $\widetilde{\lambda}^{i+1}$ and $\widetilde{\lambda}^{i}$.
 \end{itemize}

 Note that, quite remarkably, this procedure does not depend on $e$ ! this has in fact a crystal theory explanation: this process correspond to an isomorphism of crystals between $\mathcal{U} (\widehat{\mathfrak{sl}}_e)$-modules. It can be shown that the associated crystals embed in crystals of $\mathcal{U} ({\mathfrak{sl}}_{\infty})$-modules and that the 
  crystal isomorphism 
  is  nothing but the restriction of a crystal $\mathcal{U} ({\mathfrak{sl}}_{\infty})$-isomorphism.  It thus does not depend on $e$. 
 
 \end{abs}
 
 \begin{abs}
 
 We consider the Young diagram of $\ulambda$ and we color each node of the diagram with its content. We get a colored Young diagram as in the following example. One can visualize the above process in this colored Young diagram. This consist in adding to the Young diagram of  $\lambda^2$  a skew Young diagram to obtain the Young diagram of $\widetilde{\lambda}^1$ and delete it 
  to $\lambda^1$ to obtain $\widetilde{\lambda}^2$. This skew Young diagram is obtained by looking at the $\beta$-numbers we have defined above.

\end{abs}

\begin{Exa}
Take $(s_1,s_2)=(1,0)$ and the $2$-partition $(5.5.3.1,3.1)$. We here have $r_1=4$ and $r_2=2$. Thus $r=4$. 
\begin{itemize}
\item We start with $\lambda^2_3=0$. As $\beta (\lambda^2_3)=-3$ then $a_1=4$ because $\beta (\lambda^1_4)=-2$. 
 Thus we obtain $\widetilde{\lambda}^1_3=1$ and $\widetilde{\lambda}^2_4=0$.
\item For $\lambda^2_2=1$, as $\beta (\lambda^2_2)=-1$ then $a_1=3$ because $\beta (\lambda^1_3)=1$. 
 Thus we obtain $\widetilde{\lambda}^1_2=3$ and $\widetilde{\lambda}^2_3=1$.
\item For $\lambda^2_1=3$, as $\beta (\lambda^2_1)=2$ then $a_1=2$ because $\beta (\lambda^1_2)=4$. 
 Thus we obtain $\widetilde{\lambda}^1_1=5$ and $\widetilde{\lambda}^2_2=3$.
 \end{itemize}
We thus get $\widetilde{\ulambda}=(5.3.1,5.3.1)$. At the level of colored Young diagram, we get for $\ulambda$ the following one:

$$
\left(
\begin{array}{|c|c|c|c|c|}
  \hline
  \ 1 \ &\ 2\  &\ 3\  &   4\  &\ { 5}\    \\
  \hline 
  \  0 \ &\ 1\ &\ 2 \ &\  {\bf 3}\ &\ {\bf 4}\ \\
    \hline
    \  \underline{1}\ & \ {\bf 0}\ & \ {\bf 1}\   \\
\cline{1-3}
\ \underline{{\bf 2}}\ \\
\cline{1-1}
\end{array}\;,\;
\begin{array}{|c|c|c|}
  \hline
  \ 0 \ &\ 1\  &\ 2\     \\
  \hline
 \ \underline{1}\   \\
 \cline{1-1}
\end{array}
\right)$$
and we obtain for ${\widetilde{\lambda}}$:

$$
\left(
\begin{array}{|c|c|c|c|c|}
  \hline
  \ 0 \ &\ 1\  &\ 2\  &   {\bf 3}\  &\ {\bf 4}\ \\
  \hline 
    \  \underline{1}\ & \ {\bf 0}\ & \ {\bf 1}\   \\
\cline{1-3}
\ \underline{{\bf 2}}\ \\
\cline{1-1}
\end{array}\;,\;
\begin{array}{|c|c|c|c|c|}
\hline
  \  1 \ &\ 2\ &\ 3 \ &\  { 4}\ &\ {5}\ \\
    \hline
  \ 0 \ &\ 1\  &\ 2\     \\
  \cline{1-3}
 \ \underline{1}\   \\
 \cline{1-1}
\end{array}
\right)$$
where $\underline{a}$ for $a\in \mathbb{Z}_{\geq 0}$ means $-a$.

  \end{Exa}

\section{Labeling of one dimensional representations}

 Our goal is to develop the strategy of \S\ref{strate} to obtain an explicit description of our one dimensional representations.  So we assume that 
  ${\bf s}\in \mathbb{Z}^l$ and that $\ulambda \in \Lambda (n)$. If $\umu=(\mu^{1},\ldots,\mu^l)$ and 
   $\unu=(\nu^1,\ldots,\nu^l)$, we denote by $\umu+\unu$ the $l$-partition $(\mu^1+\nu^1,\ldots,\mu^l+\nu^l)$. Here, 
 the sum of two partitions $\mu=(\mu_1,\ldots,\mu_r)$ and $\nu=(\nu_1,\ldots,\nu_{r})$ is defined as $(\nu_1+\mu_1,\ldots,\nu_r+\mu_r)$ (where we add as many zero parts to the partitions $\mu$ or $\nu$ if necessary). 

\begin{abs}\label{prem}
Let ${\bf s}\in \mathbb{Z}^l$ and define  the set:
$$\mathfrak{I}:=\{ i\in \{0,1,\ldots,e-1\}\ |\ \exists j\in \{1,\ldots,l\},\ s_j\equiv i+e\mathbb{Z}\}.$$
As in  our discussion in \S\ref{disc1}, note that we have 
$$[n,k_1]_{\theta_{\bf s}}=[n,k_2]_{\theta_{\bf s}}\qquad \text{and} \qquad 
[1^n,k_1]_{\theta_{\bf s}}=[1^n,k_2]_{\theta_{\bf s}},$$
as soon as $s_{k_1}=s_{k_2}+e\mathbb{Z}$. For each $k\in \mathfrak{I}$, we define $\alpha (k)$ to be the minimal element of $\{1,\ldots,l\}$ such that 
$$s_{\alpha (k)}=\textrm{max} \{ s_j\ |\ j\in \{1,\ldots,l\},\ s_j    \equiv k+e\mathbb{Z}\}.$$
We thus have $s_{\alpha (k)}+e\mathbb{Z}=k+e\mathbb{Z}$.  
The question remains to determine $[n,\alpha (k)]$ and $[1^n,\alpha (k)]$ for all $k\in \mathfrak{I}$.

Following our strategy exposed in \S \ref{strate}, there exists  $\sigma \in \widehat{\mathfrak{S}}_l$
   such that $\sigma.{\bf s}'= {\bf s}$ with ${\bf s}' \in \mathcal{A}_{e,l}$.  We fix $k\in \mathfrak{I}$, then we obtain: 
   \begin{itemize}
   \item For $\ulambda=[n,\alpha (k) ]_{\theta_{\bf s}}$ we have by \S\ref{flows}, that    
    $$g_{{\bf s}',e} ((\sigma_0^{-1} \ulambda)_{{\theta}_{{\bf s}'}})=(k+e\mathbb{Z},k+1+e\mathbb{Z},\ldots, k+n-1+e\mathbb{Z}),$$
   and thus
     $$g_{{\bf s},e} (( \ulambda)_{{\theta}_{{\bf s}}})=(k+e\mathbb{Z},k+1+e\mathbb{Z},\ldots, k+n-1+e\mathbb{Z}).$$
    \item  For $\ulambda=[1,\alpha (k) ]_{\theta_{\bf s}}$ we have by Proposition \ref{mull} and \S\ref{exp}, 
       $$g_{{\bf s},e} (( \ulambda)_{{\theta}_{{\bf s}}})=(k+e\mathbb{Z},k-1+e\mathbb{Z},\ldots, k-n+1+e\mathbb{Z}).$$
     
   \end{itemize}
   
   We can now first  study more precisely the multipartitions $[n,\alpha (k) ]_{\theta_{\bf s}}$, we then turn to the multipartitions 
    $[1^n,\alpha (k) ]_{\theta_{\bf s}}$ for $k\in \mathfrak{I}$. 
  
\end{abs}

 \begin{Prop}\label{first} Under the above notations,  we have the following result. 
 \begin{enumerate}
 \item If for all $k_1\in \mathfrak{I}$ we have:
 $$s_{\alpha (k)}>s_{\alpha (k_1)} \text { or  } ( s_{\alpha (k_1)}> s_{\alpha (k)}>s_{\alpha (k_1)}-e \textrm{ and } \alpha (k)<\alpha (k_1)).$$
 Then we have   $[n,\alpha (k)]_{\theta_{\bf s}}=[n,\alpha (k)]$.
\item Otherwise,  let ${k_1}\in \mathfrak{I}$ be such that: 
\begin{itemize}
\item $s_{\alpha (k_1)}>s_{\alpha (k)}+e$ or  ($s_{\alpha (k_1)}>s_{\alpha (k)}>s_{\alpha (k_1)}-e$ and $\alpha (k_1)<\alpha (k)$.)
\item The integer  $j_{k_1}\in \{1,\ldots e-1\}$ such that $j_{k_1} \equiv k_1-k+e\mathbb{Z}$ is minimal with the above property. 
\end{itemize}
We then have:
$$[n,\alpha (k)]_{\theta_{\bf s}}=
\left\{ 
\text{ \begin{tabular}{cc}
$[n,\alpha (k)]$ & { if } $n\leq j_{k_1}$, \\
$[j_{k_1},\alpha (k)]+[n-j_{k_1},\alpha (k_1)]_{\theta_{\bf s}}$ &{ otherwise.}
\end{tabular}}\right.$$
 \end{enumerate} 
 \end{Prop}

\begin{proof}

We have to consider the sequence $(k+e\mathbb{Z},k+1+e\mathbb{Z},\ldots, k+n-1+e\mathbb{Z})$ and add successively good nodes with appropriate residues, starting with the empty multipartition. 

\begin{enumerate}

\item Assume that we are in the first case of the theorem.  Assume moreover that   $j\in \mathfrak{I}$  is such that there exists  $t\in \{1,\ldots, n-1\}$ satisfying:
 $$k+t+e\mathbb{Z}=j+e\mathbb{Z}.$$
 Then, in the $l$-partition $[t,\alpha (k)]$,   we have: 
 $$ (1,1,\alpha (j)) \prec_{{\bf s},e}  (1,t+1,\alpha (k)),$$ 
 because of our hypothesis and because of the definition of  $\prec_{{\bf s},e}$.
Thus, for this $l$-partition,  the good $k+t+e\mathbb{Z}$-node is $(1,t+1,\alpha (k))$.  The result follows. 

\item Consider now the second case. If $n\leq j_{k_1}$ then it is clear that the good nodes are successively  added  in the component $\alpha (k)$ of the $l$-partition and thus we have $[n,\alpha (k)]_{\theta_{\bf s}}=[n,\alpha (k)]$. Otherwise, considering the sequence of residue:
 $$( k+e\mathbb{Z},k+1+e\mathbb{Z},\ldots, k+j_{k_1}-1+e\mathbb{Z}).$$
The  $l$-partition which is obtained by adding the good nodes associated to this sequence is 
 $[j_{k_1},\alpha (k)]$. Now note that the removable node $(1,j_{k_1},\alpha (k))$ comes with a residue:
 $$k+j_{k_1}-1+e\mathbb{Z}.$$
If we now want to add a good addable node with residue $k+j_{k_1}+e\mathbb{Z}$, we have to add it in the component $\alpha (k_1)$ by hypothesis. Note then that the node $(2,1,\alpha (k_1))$ becomes an addable node with residue:
 $$1-2+k_1+e\mathbb{Z}=k+j_{k_1}-1+e\mathbb{Z}.$$
This node is greater than $(1,j_{k_1},\alpha (k))$ and there are no (and there will be no) addable or removable nodes with the same residue between these two nodes. 
 Thus in the process of computing good nodes, these two nodes have no impact in the following (because it is associated to an occurrence  $RA$ in the notation of \S\ref{good}). We now have to consider the sequence: 
 $$(k+j_k+e\mathbb{Z},k+j_k+1+e\mathbb{Z},\ldots,k+n-1+e\mathbb{Z}).$$
 and add good nodes with the appropriate residues to our multipartition $[j_{k_1},\alpha (k)]$. Note that in this process, we will never add new nodes in the component $\alpha (k)$ of our $l$-partition. 
 We can then conclude by induction. 

\end{enumerate}

\end{proof}

\begin{Exa}\label{exa1} Set $e=4$ and ${\bf s}=(3,0,7,3)$ then we have $\mathfrak{I}=\{0,3\}$ and we have $\alpha (0)=2$ and $\alpha  (3)=3$. Following the above theorem we thus have 
$$[n,2]_{\theta_{\bf s}}=\left\{ \begin{array}{ccc}
 (\emptyset,n,\emptyset,\emptyset) & \text{if} & n\leq 3, \\
 (\emptyset,3,n-3,\emptyset) & \text{otherwise.} &
\end{array}\right.$$
and combining Proposition \ref{triv} with the above theorem gives
$$[n,3]_{\theta_{\bf s}}=[n,1]_{\theta_{\bf s}}=[n,4]_{\theta_{\bf s}}= (\emptyset,\emptyset,n,\emptyset).$$
\end{Exa}
We now turn to the representations of type $[1^n,i]_{\theta_{\bf s}}$. 
\begin{abs}\label{part} In general, the result is more complicated to describe than for the other  type of $1$-dimensional representations and it requires to much notations for
  being stated and useful in the general setting. Instead, we will use the following idea. Let $k_0,\ldots,k_s$ be the element in $\mathfrak{I}$ such that $\alpha (k_0)<\alpha (k_1)<\ldots<\alpha (k_s)$. We here make the assumption that we have:
  $$s_{\alpha (k_0)}>s_{\alpha (k_1)}>\ldots >s_{\alpha (k_s)}.$$
We will explain later how we can deduce the general result from this case. 
Let $i\in \{0,1,\ldots,s\}$, we want to find 
   $[1^n,\alpha (k_i)]_{\theta_{\bf s}}$. 
   \end{abs}
  \begin{Prop}\label{resup}
  Under the above hypotheses, we have 
  $$[1^n,\alpha (k_0)]_{\theta_{\bf s}}=[ \left( (q+1)^r,q^{e-1-r} \right)   ,\alpha (k)],$$
   where $(q,r)\in \mathbb{Z}_{>0}\times \{0,1,\ldots,e-2\}$ is uniquely defined by the euclidean division $n=q (e-1)+r$. If $i\neq 0$, we consider $j\in \{0,\ldots,i-1\}$ such that $x\in \{1,\ldots,e-1\}$ is minimal such that $x+e\mathbb{Z}=k_i-k_j+e\mathbb{Z}$. 
 Then we have: 
  $$[1^n,\alpha (k_i)]_{\theta_{\bf s}}=
\left\{ 
\text{ \begin{tabular}{cc}
$[1^n,\alpha (k_i)]$ & { if } $n\leq x$, \\
$[1^x,\alpha (k_i)]+[1^{n-x},\alpha (k_j)]_{\theta_{\bf s}}$ &{ otherwise.}
\end{tabular}}\right.$$
  
  \end{Prop}
 \begin{proof}
In the case where $i=0$, the problem reduces to the case where  $l=1$. Indeed, the sequence of residues we have to consider is 
$$(k_0+e\mathbb{Z},k_0-1+e\mathbb{Z},\ldots, k_0-n+1+e\mathbb{Z}).$$
Because of our assumption $s_{\alpha (k_0)}>s_{\alpha (k_1)}>\ldots >s_{\alpha (k_s)}$, we see that the good nodes we have to add are always located  in the component $\alpha (k_0)$ of our $l$-partition because the greatest addable nodes with appropriate residues are always in this component. The result follows in this case.

Assume now that $i\neq 0$. Consider 
 $j\in \{0,\ldots,i-1\}$  such that $x\in \{1,\ldots,e-1\}$ is minimal such that $x+e\mathbb{Z}=k_i-k_j+e\mathbb{Z}$. 
 Consider our sequence of residues:
  $$(k_i+e\mathbb{Z},k_i-1+e\mathbb{Z},\ldots, k_i-n+1+e\mathbb{Z}).$$
 For the sequence 
 $$(k_i+e\mathbb{Z},k_i-1+e\mathbb{Z},\ldots, k_i-x+1+e\mathbb{Z}),$$
 the addable good node appears  in the component $\alpha (k_i)$ of the multipartition. 
  We thus obtain the $l$-partition $[ 1^{x}  ,\alpha (k_i)]$. Now we must note that we have a removable  node $R$
  with residue $k_i-x+1+e\mathbb{Z}$ on the component $\alpha (k_i)$. To continue our process, we need to add a good node with residue $k_i-x+e\mathbb{Z}=k_j+e\mathbb{Z}$ and thus consider the sequence 
  $$(k_j+e\mathbb{Z},k_j-1+e\mathbb{Z},\ldots, k_j-n+1-x+e\mathbb{Z}).$$
To conclude by induction, we have to check that our removable  node $R$ does not change the possible good nodes. But as we add the good $k_j+e\mathbb{Z}$-node in the component $\alpha (k_j)$, we get an addable 
 $k_i-x+1+e\mathbb{Z}$-node in the component  $\alpha (k_j)$ which is greater than $R$ and there will be no addable or removable node with the same residue between these two in the following. We thus obtain an occurrence $RA$ in the process of computing the good nodes. This thus doesn't interfere in the process. We can thus conclude by induction  as in the proof of the last proposition. 
 \end{proof}
 
 \begin{abs}\label{fin}
Now, how can we deduce the result for the general case ? to do that we can use the description of the ``crystal isomorphism'' given in section \ref{crystal}. Assume that ${\bf s}\in \mathbb{Z}^l$. Then there exists $\sigma\in \mathfrak{S}_l$ such that ${\bf s}'=\sigma. {\bf s}$ satisfies the condition in \S \ref{part}.  We can thus use Proposition \ref{resup} to obtain 
     $[1^n,\alpha (k_i)]_{{\theta_{{\bf s}'}}}$ for all $0\leq i\leq s$ and then conclude using \S \ref{algo}.
    
    \end{abs}
    \begin{Exa}
    Set $e=4$ and ${\bf s}=(3,0,7,3)$. Assume that $n>1$  then we have $\mathfrak{I}=\{0,3\}$ and we have $\alpha (0)=2$ and $\alpha  (3)=3$
     and $s_2=0$ and $s_3=7$.
    Let us consider ${\bf s} ' =(7,0,3,3)$. Then we are in position to apply Proposition \ref{resup} to $\theta_{{\bf s}'}$
    $$(1^n,\emptyset, \emptyset,\emptyset)_{{\theta_{\bf s}}'}=   ( (q+1)^r.q^{e-1-r}, \emptyset,\emptyset,\emptyset ),$$
    where  $(q,r)\in \mathbb{Z}_{>0}\times \{0,1,\ldots,e-2\}$ is uniquely defined by the euclidean division $n=q (e-1)+r$. 
    Now we have:
        $$(\emptyset, 1^n,\emptyset,\emptyset)_{{\theta_{\bf s}}'}=   ( (q'+1)^{r'}.{q'}^{e-1-{r'}}, 1,\emptyset,\emptyset ),$$
        where  $(q',r')\in \mathbb{Z}_{>0}\times \{0,1,\ldots,e-2\}$ is uniquely defined by the euclidean division $n-1=q' (e-1)+r'$.
        To conclude now, we have to consider the isomorphism $\chi_{e,{\bf s},{\bf s}'}$ which we have described in \S\ref{algo}. Applying it to 
          our two $4$-partitions will give what we wanted that is 
        $(1^n,\emptyset, \emptyset,\emptyset)_{{\theta_{\bf s}}}=(\emptyset, \emptyset,1^n,\emptyset)_{{\theta_{\bf s}}}=(\emptyset, \emptyset,\emptyset,1^n)_{{\theta_{\bf s}}}$ and $(\emptyset, 1^n,\emptyset,\emptyset)_{{\theta_{\bf s}}'}$. 
     Note that  in our case, we can take  $\sigma=\sigma_1\sigma_2\sigma_1$.

    \end{Exa}
    
 \section{Application to type $B_n$}
 
 In this part, we apply the above result to the case of Hecke algebra of type $B_n$ that is the case $l=2$ where we give closed formulae for each case. So we assume that $(s_1,s_2)\in \mathbb{Z}^2$ and we will study three different cases.
 
 \begin{abs}
 First let us assume that $\underline{s_1=s_2}$. Then by Proposition \ref{triv}, we obtain:
 $$((n),\emptyset)_{\theta_{\bf s}}=(\emptyset,(n))_{\theta_{\bf s}}=((n),\emptyset),$$
 and from Proposition \ref{resup}, we obtain:
 $$((1^n),\emptyset)_{\theta_{\bf s}}=(\emptyset,(1^n))_{\theta_{\bf s}}=(m_e (1^n),\emptyset),$$ 
 where $m_e$ is the Mullineux map in type $A$ given in Example \ref{typA}.
 \end{abs}
\begin{abs}
Now let us assume that $\underline{s_1>s_2}$. Then by Proposition \ref{first}, we obtain:
 $$((n),\emptyset)_{\theta_{\bf s}}=((n),\emptyset),$$
and if we set $j\in \{0,\dots,e-1\}$ such that $j+e\mathbb{Z}=s_1-s_2+e\mathbb{Z}$.
 $$(\emptyset,(n))_{\theta_{\bf s}}=
 \left\{
 \begin{array}{ccc}
 (\emptyset,(n)) & \text{ if } j\leq n, \\
 ((n-j),(j)) & \text{ otherwise.}
 \end{array}\right.
 $$
 By Proposition \ref{resup}, we obtain:
  $$((1^n),\emptyset)_{\theta_{\bf s}}=(m_e(1^n),\emptyset),$$
 again where $m_e$ is the Mullineux map in type $A$ 
 and if we set $j\in \{0,\dots,e-1\}$ such that $j+e\mathbb{Z}=s_2-s_1+e\mathbb{Z}$.
 $$(\emptyset,(1^n))_{\theta_{\bf s}}=
 \left\{
 \begin{array}{ccc}
 (\emptyset,(1^n)) & \text{ if } j\leq n, \\
 (m_e (1^{n-j}),(1^j)) & \text{ otherwise.}
 \end{array}\right.
 $$

\end{abs}
\begin{abs}
Finally let us study the case where $\underline{s_2>s_1}$. 
Then by Proposition \ref{first}, we obtain:
 $$(\emptyset, (n))_{\theta_{\bf s}}=(\emptyset,(n)),$$
and if we set $j\in \{0,\dots,e-1\}$ such that $j+e\mathbb{Z}=s_2-s_1+e\mathbb{Z}$.
 $$((n),\emptyset)_{\theta_{\bf s}}=
 \left\{
 \begin{array}{ccc}
 ((n),\emptyset) & \text{ if } j\leq n, \\
 ((j), (n-j)) & \text{ otherwise.}
 \end{array}\right.
 $$
Now let us consider the last two types of representations that we want to characterize. For this, we can  directly use \S \ref{fin}  or argue as follows: 
\begin{itemize}
\item Assume first that we have $s_2\geq s_1+e$.  Then we will use the bijection 
 described in \S \ref{algo0}. 
We have that  ${\bf s}'=(s_2,s_1+e)$ satisfies $s_2\geq s_1+e$ so we can use the case above  to deduce that:
  $$(\emptyset,(1^n))_{\theta_{{\bf s}'}}=(m_e(1^n),\emptyset),$$
 again where $m_e$ is the Mullineux map in type $A$. We thus obtain:
 $$(\emptyset,(1^n))_{\theta_{{\bf s}}}=(\emptyset,m_e(1^n)).$$
Now,  if we set $j\in \{0,\dots,e-1\}$ such that $j+e\mathbb{Z}=s_1+e-s_2+e\mathbb{Z}$, we obtain:
 $$(\emptyset, (1^n) )_{\theta_{{\bf s}'}}=
 \left\{
 \begin{array}{ccc}
 (\emptyset,(1^n)) & \text{ if } j\leq n, \\
 ((m_e (1^{n-j}),(1^j)) & \text{ otherwise.}
 \end{array}\right.
 $$
and by \S \ref{algo0} we conclude:
 $$((1^n),\emptyset )_{\theta_{\bf s}}=
 \left\{
 \begin{array}{ccc}
 ((1^n),\emptyset) & \text{ if } j\leq n, \\
 ((1^j),m_e (1^{n-j})) & \text{ otherwise.}
 \end{array}\right.
 $$
\item Assume finally that $s_1+e>s_2>s_1$. We want to find  $(\emptyset,(1^n))_{\theta_{\bf s}}$.  To do this, we just have to apply the strategy proposed  in \S \ref{fin}  and consider $(1^n,\emptyset)_{\theta_{(s_2,s_1)}}$. We need thus to apply our algorithm to 
$(m_e (1^n),\emptyset)$.  After a quick calculation, we see that: 

$$(\emptyset,(1^n))_{\theta_{\bf s}}=\left\{ \begin{array}{ccc}
((q+1)^{r-(s_2-s_1)}q^{e-1-r} , (q+1)^{s_2-s_1}) & \text{ if } & s_2-s_1\leq r,\\
(q^{e-1-s_2+s_1} , (q+1)^{r}q^{s_2-s_1-r}) & \text{ otherwise.} & \\
\end{array}\right.$$
where $n=q(e-1)+r$ for $q\in \mathbb{Z}_{>0}$ and $0\leq r\leq e-1$. 

Now for the case $(1^n,\emptyset)_{\theta_{\bf s}}$, we again use the bijection of \S \ref{algo0}.  Let us consider 
$(\emptyset,1^n)_{\theta_{(s_2,s_1+e)}}$. We have $s_2+e> s_1+e>s_2$ so by the above result we have: 
$$(\emptyset,(1^n))_{\theta_{(s_2,s_1+e)}}=\left\{ \begin{array}{ccc}
((q+1)^{r-(s_1+e-s_2)}q^{e-1-r} , (q+1)^{s_1+e-s_2}) & \text{ if } & s_1+e-s_2\leq r,\\
(q^{e-1-s_1-e+s_2} , (q+1)^{r}q^{s_1+e-s_2-r}) & \text{ otherwise.} & \\
\end{array}\right.$$
and we can thus conclude  that 
$$((1^n),\emptyset)_{\theta_{(s_1,s_2)}}=\left\{ \begin{array}{ccc}
( (q+1)^{s_1+e-s_2},(q+1)^{r-(s_1+e-s_2)}q^{e-1-r} ) & \text{ if } & s_1+e-s_2\leq r,\\
( (q+1)^{r}q^{s_1+e-s_2-r},q^{s_2-1-s_1} ) & \text{ otherwise.} & \\
\end{array}\right.$$

\end{itemize}

\end{abs}
%
%
%
%

\begin{Rem}
The explicit determination of the parameters $(s_1,s_2)$ associated to a weight function on a   Hecke algebra of type $B_n$ is explained in \cite[\S 6.7.5]{GJ}.

\end{Rem}

\vspace{0.5cm}
\noindent {\bf Address:}\\

\noindent \textsc{Nicolas Jacon}, Universit\'e de Reims Champagne-Ardenne, UFR Sciences exactes et naturelles, Laboratoire de Math\'ematiques EA 4535
Moulin de la Housse BP 1039, 51100 Reims, FRANCE\\  \emph{nicolas.jacon@univ-reims.fr}


\begin{thebibliography}{99}                                                                                               %


\bibitem{A}
\textsc{Ariki, S.}
Representations of Quantum Algebras and Combinatorics of Young Tableaux, University Lecture Series, AMS 2012 Volume 26.


\bibitem{CJ}
\textsc{Chlouveraki, M. and  Jacon, N.}
 Schur elements and basic sets for cyclotomic Hecke algebras. J. Algebra Appl. 10 (2011), no. 5, 979-993.

\bibitem{Ma}
\textsc{Fayers, M},
 Weights of multipartitions and representations of Ariki-Koike algebras. II. Canonical bases. J. Algebra 319 (2008), no. 7, 2963-2978.


\bibitem{ST}
\textsc{Geck, M.},
On the $\ell$-modular composition factors of the Steinberg representation. 
 {\bf arXiv:1504.04157}.


\bibitem{GJ}
\textsc{Geck, M. and  Jacon, N.},
 Representations of Hecke algebras at roots of unity. Algebra and Applications, 15. Springer-Verlag London, Ltd., London, 2011.

\bibitem{GP}
\textsc{Geck, M  and  Pfeiffer, G.},
Characters of finite Coxeter groups and Iwahori-Hecke algebras. London Mathematical Society Monographs. New Series, 21. The Clarendon Press, Oxford University Press, New York, 2000. 

\bibitem{Ge1}
\textsc{Gerber, T.},
Generalised canonical basic sets for Ariki-Koike algebras.
Volume 413, 1 September 2014, Pages 364-401.



\bibitem{JL}
\textsc{Jacon, N. and Lecouvey, C.}, 
On the Mullineux involution for Ariki-Koike algebras. J. Algebra 321 (2009), no. 8, 2156-2170.


\bibitem{James}
\textsc{James G.D. },
The decomposition matrices of $\text{GL}_n (q)$ for $n \leq 10$, Proc. LMS (3), 60 (1990),
225-264.





\end{thebibliography}
\end{document}